\documentclass[twoside]{amsart}
\usepackage{latexsym}
\usepackage{amssymb,amsmath,amsopn}
\usepackage[dvips]{graphicx}   
\usepackage{color,epsfig}      
\usepackage{tikz} 
\usepackage{subcaption} 
\usepackage{array}
\usepackage{hyperref}
\usepackage{mathtools}
\newtheorem{thm}{Theorem}

\newtheorem{lem}{Lemma}

\theoremstyle{definition}
\newtheorem{defn}{Definition}

\newtheorem{conj}{Conjecture}

\renewcommand{\Re}{\mathbb R}

\newcommand{\BB}{\mathbf B}
\newcommand{\M}{\mathcal{M}}

\newcommand{\F}{\mathcal{F}}
\newcommand{\PP}{\mathcal{P}}

\newcommand{\X}{\mathcal{X}}

\DeclareMathOperator{\Id}{Id}

\DeclareMathOperator{\conv}{conv}

\DeclareMathOperator{\proj}{proj}
\DeclareMathOperator{\perim}{perim}
\DeclareMathOperator{\area}{area}

\DeclareMathOperator{\skel}{skel}
\DeclareMathOperator{\surf}{surf}

\DeclareMathOperator{\vol}{vol}

\begin{document}

\title[On edge density]{On the edge densities of normal, convex mosaics}

\author[M. Kadlicsk\'o]{M\'at\'e Kadlicsk\'o}
\author[Z. L\'angi]{Zsolt L\'angi}
\author[S. Lyu]{Shanxiang Lyu}

\address{M\'at\'e Kadlicsk\'o, Independent researcher, Budapest, Hungary,\\
M\H uegyetem rkp. 3, H-1111 Budapest, Hungary}
\email{kadlicsko.mate@gmail.com}
\address{Zsolt L\'angi, Bolyai Institute, University of Szeged, Aradi v\'ertan\'uk tere 1, H-6720 Szeged, Hungary, and\\
Alfr\'ed R\'enyi Institute of Mathematics, Re\'altanoda utca 13-15, H-1053, Budapest, Hungary}
\email{zlangi@server.math.u-szeged.hu}
\address{Shanxiang Lyu, College of Cybersecurity, Jinan University, Guangzhou 510632, China}
\email{lsx07@jnu.edu.cn}

\thanks{Partially supported by the National Research, Development and Innovation Office, NKFI, K-147544 grant, the ERC Advanced Grant “ERMiD”, the Project no. TKP2021-NVA-09 with the support provided by the Ministry of Innovation and Technology of Hungary from the National Research, Development and Innovation Fund and financed under the TKP2021-NVA funding scheme, the National Natural Science Foundation of China under Grant 62571218, and the Guangdong Basic and Applied Basic Research Foundation under Grant 2024A1515030089.}


\begin{abstract}
In this paper we investigate the problem of finding the minimum edge density in families of convex, normal mosaics with unit volume cells in $n$-dimensional Euclidean space. In the first part of the paper we solve this problem for mosaics whose cells are Minkowski sums of cells of $1$ or $2$-dimensional mosaics. We show that while for $n=2$ this minimum is attained by a mosaic with regular hexagon cells, this is not true in any dimension $n > 2$, where the minimum is attained by a mosaic whose cells are Minkowski sums of pairwise orthogonal regular triangles, and possibly a segment. In the second part we investigate $3$-dimensional convex mosaics whose cells are translates of a given convex polyhedron, and show that within this family, mosaics with cubes as cells have minimum edge density. In addition, using our method, in the family of $3$-dimensional convex polyhedra whose translates tile the space,  we find the unit volume polyhedra with minimal total edge length.
\end{abstract}

\maketitle

\section{Introduction}

The aim of this paper is to investigate certain properties of mosaics. A \emph{mosaic} or \emph{tiling} of the Euclidean $n$-space $\Re^n$ is a countable family $\M$ of compact sets, called cells, with the property that $\bigcup \M = \Re^n$, and the interiors of any two distinct cells are disjoint. A mosaic is \emph{convex}, if every cell in it is convex; it is known that in this case every cell is a convex polytope. If, for a mosaic $\M$, there are universal constants $0 < r < R$ such that the inradius of every cell of $\M$ is at least $r$, and its circumradius is at most $R$, the mosaic is called \emph{normal}. In this paper we deal only with normal, convex mosaics. A convex mosaic is called \emph{face-to-face} (in the plane also \emph{edge-to-edge}) if every facet of every cell is a facet of exactly one more cell. Furthermore, a convex mosaic is \emph{translative} (resp. \emph{congruent}), if every cell is a translate (resp. a congruent copy) of a given convex polytope. Clearly, every congruent and translative convex mosaic is normal. A cell of a translative convex mosaic in $\Re^n$ is called an \emph{$n$-dimensional parallelohedron}.

One of the classical problems of geometry, regarding planar mosaics, is the Honeycomb Conjecture, stating that in a decomposition of the Euclidean plane into cells of equal area, the regular hexagonal grid has minimal edge density. This conjecture first appeared in Roman times in a book of Varro about agriculture \cite{Varro}. The problem has been the center of research throughout the second half of the 20th century \cite{LFT, Morgan}; it was solved by Fejes T\'oth for normal, convex mosaics, while the most general version is due to Hales \cite{Hales1}, who dropped the condition of convexity.

A famous analogous problem for mosaics in the $3$-dimensional Euclidean space $\Re^3$ was proposed by Lord Kelvin \cite{Kelvin} in 1887; he conjectured that in a tiling of space with cells of unit volume, minimal surface density is attained by a mosaic composed of slightly modified truncated octahedra. Even though this conjecture was disproved by Weaire and Phelan in 1994 \cite{Weaire}, who discovered a tiling of space with two slightly curved `polyhedra' of equal volume and with less surface density than in Lord Kelvin's mosaic, the original problem of finding the mosaics with equal volume cells and minimal surface density has been extensively studied (see, e.g. \cite{Kusner, Gabbrielli, Oudet}). On the other hand, in the authors' knowledge, there is no subfamily of mosaics for which Kelvin's problem is solved. Here, in particular, we should recall a yet unsolved conjecture of Bezdek \cite{Bezdek} from 2006, restated in \cite{Langi} in 2022, that states that among translative, convex mosaics with unit volume cells, the ones generated by regular truncated octahedra have minimal surface density.

In this paper, we investigate a similar problem. The main definition of our paper is the following.

\begin{defn}\label{defn:density}
Let $\M$ be a convex mosaic in $\Re^n$. We define the \emph{upper edge density} of $\M$ as the limit
\begin{equation}\label{eq:updense}
\overline{\rho}_1(\M) = \limsup_{R \to \infty} \frac{L(\skel_1(\M) \cap R \BB^n)}{\vol(R \BB^n)},
\end{equation}
if it exists, where $\BB^n$ is the closed unit ball centered at the origin $o$, the $1$-skeleton $\skel_1(\M)$ of $\M$ is defined as the union of all edges of all cells of $\M$, $L(\cdot)$ denotes length, and $\vol(\cdot)$ denotes volume. Replacing $\limsup$ with $\liminf$ in (\ref{eq:updense}) we obtain the \emph{lower edge density} $\underline{\rho}_1(\M)$ of $\M$. If $\overline{\rho}_1(\M) = \underline{\rho}_1(\M)$, we say that this common value is the \emph{edge density} of $\M$.
\end{defn}

We note that in our definition of the average of a geometric quantity of a tiling in Euclidean space, we follow the classical approach, described e.g. in \cite{Lagerungen} and applied in many discrete geometry papers. For a more general approach, see e.g. \cite{BR02, BR04}. 

It is natural to ask about the minimal edge density of convex mosaics in $\Re^n$. In light of the conjectures of Lord Kelvin and Bezdek in the previous paragraph, answering this question for any dimension $n \geq 3$ seems to be a very difficult problem, well beyond the scope of this paper. Thus, we restrict our investigation to special types of mosaics.

\begin{defn}\label{defn:decomposable}
Let $L_1, L_2, \ldots, L_k$ be linear subspaces of dimensions $n_1, n_2, \ldots, n_k$, respectively, such that $\sum_{i=1}^k n_i = n$ and $\Re^n = \sum_{i=1}^k L_i$.
For $i=1,2\ldots, k$, let $\M_i$ be a convex mosaic in $L_i$. We say that a convex mosaic $\M$ in $\Re^n$ can be \emph{decomposed} into the convex mosaics $\M_1, \M_2, \ldots, \M_k$ (or that $\M$ is the \emph{sum} of $\M_1, \M_2, \ldots, \M_k$), if a set $C$ is a cell of $\M$ if and only if $C= \sum_{i=1}^k C_i$ for some cells $C_i$ of $\M_i$, where $i=1,2, \ldots, k$.
\end{defn}

In the paper, we denote by $\F_{dc}^n$ the family of convex, normal mosaics in $\Re^n$ that can be decomposed into convex mosaics of dimensions at most $2$. Our first result is the following.

\begin{thm}\label{thm:decomp}
Let $n \geq 2$, and let $\M \in \F_{dc}^n$. Then
\begin{equation}
\underline{\rho}_1(\M) \geq
\left\{
\begin{array}{ll}
\sqrt{2 \sqrt{3}}, & \hbox{ if } n=2;\\
\frac{ \sqrt{3 \sqrt{3}} k}{2^{k-1}}, & \hbox{ if } n=2k \hbox{ for some integer } k \geq 2;\\
\frac{3\sqrt{3} k^{2/3}}{2^k}, & \hbox{ if } n=2k+1 \hbox{ for some integer } k \geq 1. 
\end{array}
\right.
\end{equation}
Furthermore, in the above inequality we have equality
\begin{enumerate}
\item[(i)] for $n=2$, if the cells of $\M$ are unit area regular hexagons;
\item[(ii)] for $n=2k$, where $k \geq 2$, if $\M$ is decomposed into $k$ mosaics in mutually orthogonal planes so that the cells in each of them are unit area regular triangles;
\item[(iii)] for $n=2k+1$, where $k \geq 1$, we have equality if $\M$ is decomposed into $k$ planar and $1$ one-dimensional mosaics in mutually orthogonal linear subspaces satisfying the following: the cells of the one-dimensional mosaic are segments of length $l=\sqrt{3}k^{2/3}$, and the cells of the planar mosaics are regular triangles of area $\frac{1}{ 3^{1/(2k)} k^{2/3k}}$.
\end{enumerate}
\end{thm}

The main tool of the proof of Theorem~\ref{thm:decomp} is a proof of the Honeycomb Conjecture for normal, convex mosaics, presented in \cite{LW2025}.

In the second part of the paper we examine the case $n=3$. In this case we find a solution only in the family of translative, convex mosaics. To state our result, we denote the family of translative, convex mosaics in $\Re^3$ by $\F_{tr}^3$. We note that the next natural step, namely to find a solution on the family of congruent, convex mosaics seems significantly more complicated. Here we should note that the list of $3$-dimensional convex parallelohedra was known in the 19th century, whereas even for the very special case of tetrahedra it is not known which of them can be the cells of a congruent, convex mosaic of $\Re^3$. For more information on the last problem, see e.g. \cite{BDKS, Edmonds, GGHMMPW} or the survey \cite{Senechal}.

\begin{thm}\label{thm:1}
For any $\M \in \F_{tr}^3$,
\[
\underline{\rho_1}(\M) \geq 3,
\]
with equality if and only if $\M$ is a face-to-face mosaic with unit cubes as cells.
\end{thm}

The proof of this theorem is entirely different from that of Theorem~\ref{thm:decomp}, and is based on a method introduced in \cite{Langi}.
This method seems to work only for $3$-dimensional, translative mosaics: to generalize it for other dimensions or for mosaics with less restrictive conditions it seems a new approach is needed. In addition, comparing the result of Theorem~\ref{thm:1} to that of Theorem~\ref{thm:decomp}, one can see that a normal, convex mosaic in $\Re^3$ with minimal edge length cannot be a translative mosaic. Nevertheless, as isoperimetric problems for $3$-dimensional parallelohedra are natural objects of research (see e.g. \cite{Bezdek, Langi}), we believe that Theorem~\ref{thm:1} is interesting on its own.
On the other hand, the method of the proof of Theorem~\ref{thm:1}, which is based on a generalization of the method described in \cite{Langi}, can be used to prove the following theorem, where $L(P)$ denotes the total edge length of the $3$-dimensional convex polyhedron $P$.

\begin{thm}\label{thm:totaledgelength}
For any $3$-dimensional parallelohedron $P$ of unit volume, we have
\[
L(P) \geq 12,
\]
with equality if and only if $P$ is a unit cube.
\end{thm}

Theorem~\ref{thm:totaledgelength} is a proof of a special case of a conjecture of Melzak \cite{Melzak} made in 1965, claiming that among convex polyhedra of unit volume, regular triangle based prisms, with equal-length edges, have minimal total edge length. This conjecture is still very far from settled; for the solution of this problem in smaller families of convex polyhedra and for bounds on this quantity, see e.g. \cite{Abe63, Abe73, Be57, FejesToth, Melzak2}.

Based on our results, we make the following conjecture.

\begin{conj}\label{conj:edgedensity}
For any normal, convex mosaic $\M$ in $\Re^3$ with unit volume cells, we have
\[
\underline{\rho}_1(\M) \geq \frac{3\sqrt{3}}{2}.
\]
Here, we have equality e.g. if the cells of $\M$ are the regular triangle based right prisms described in Theorem~\ref{thm:decomp}.
\end{conj} 

The structure of our paper is as follows. In Section~\ref{sec:decomp} we prove Theorem~\ref{thm:decomp}. The method to prove Theorems~\ref{thm:1} and \ref{thm:totaledgelength} requires more preparation. Thus, we collect the preliminary information for their proofs in Section~\ref{sec:prelim}. In particular, here we state Theorem~\ref{thm:2}, which is a common generalization of the main result, Theorem 1, of \cite{Langi} as well as Theorems~\ref{thm:1} and \ref{thm:totaledgelength}. In Section~\ref{sec:thm1}  we deduce Theorems~\ref{thm:1} and \ref{thm:totaledgelength} from Theorem~\ref{thm:2}. Finally, in Section~\ref{sec:thm2} we prove Theorem~\ref{thm:2}.

\section{Proof of Theorem~\ref{thm:decomp}}\label{sec:decomp}

Let $\M \in \F_{dc}$. Since the sum of two $1$-dimensional mosaics is a $2$-dimensional mosaic, we may assume that if $n=2k$ is even, then $\M$ is decomposed into $k$ planar mosaics, and if $n=2k+1$ is odd, then $\M$ is decomposed into $k$ planar and one $1$-dimensional mosaics.

Let $A$ be the affine hull of any component $\M_i$ in the decomposition of $\M$, and let $\M_j$ be any other component. Let $\proj_i : \Re^n \to A^{\perp}$ denote the orthogonal projection onto the orthogonal complement $A^{\perp}$ of $A$. Then $\proj_i \M_j$ is a mosaic in a linear subspace of $A^\perp$ of the same dimension as that of $\M_i$. Even more, the sum $\M'$ of the mosaics $\M_i$, and $\proj_i \M_j$, $j \neq i$ is a mosaic in $\F_{dc}$ such that, as the length of a segment does not increase under orthogonal projection, $\underline{\rho}_1 (\M') \leq \underline{\rho}_1(\M)$. In other words, we may assume that the components of $\M$ are contained in mutually orthogonal linear subspaces. We assume that each of these subspaces is the linear hull of some of the standard basis vectors of $\Re^n$.

\emph{Case 1}: $n=2k$ is even.\\
Let $\M$ be the sum of the planar mosaics $\M_1, \M_2, \ldots, \M_k$. By the properties of decomposition and since every cell in $\M$ has unit volume, we may assume that for every $i$, $\M_i$ contains equal area cells. For $i=1,2,\ldots, k$, let $a_i$ denote the area of the cells of $\M_i$. Similarly, since $\M$ is a normal mosaic, it follows that for every $i$, $\M_i$ is a normal mosaic. We permit the cells of $\M_i$ to have straight angles, and in this way we can regard $\M_i$ as an edge-to-edge mosaic.

Let $E$ be an edge of $\M$. Then, there is a cell $P$ of $\M$ such that $E$ is an edge of $P$. Thus, by the definition of face, there is some unit vector $u \in \mathbb{S}^{n-1}$ such that $E= P \cap \{ x \in \Re^n: \langle x, u \rangle = h_P(u) \}$, where $h_P$ denotes the support function of $P$. By the properties of Minkowski sums and support functions, this condition implies that for some $i \in \{ 1,2, \ldots, k \}$, $E = E_i + \sum_{j=1, j\neq i} v_i$, where $E_i$ is an edge of $\M_i$, and $v_j$ is a vertex of $\M_j$. Equivalently, any such set is an edge of $\M$. Furthermore, observe that $\underline{\rho}_1 (\M)$ is the smallest number $\rho$ such that for any $\varepsilon > 0$, $\frac{L(\skel_1(\M) \cap R \BB^n)}{\vol(R \BB^n)} \geq \rho-\varepsilon$ for every sufficiently large value of $R$. Combining these observations, an elementary consideration shows that
\begin{equation}\label{eq:liminf_even}
\underline{\rho}_1 (\M) = \liminf_{R \to \infty} \frac{L\left(\skel_1(\M) \cap \left( R \bigtimes_{i=1}^k \mathbb{B}_i^2 \right) \right)}{\vol\left( R \bigtimes_{i=1}^k \mathbb{B}_i^2 \right))},
\end{equation}
where $\mathbb{B}_i^2$ denotes the closed $2$-dimensional unit disk, centered at the origin $o$, and lying in the plane of $\M_i$.

Now we investigate only the component $\M_i$ for some arbitrary fixed value of $i$; this consideration can be found, in more detail, also in \cite{LW2025}.
Let $s$ and $N_i(R)$ denote the numbers of cells and vertices of $\M_i$ in $R \BB^2_i$, respectively. Let the cells of $\M_i$ in $R \BB^2_i$ be $F_1, F_2, \ldots, F_s$, and let $e_j$ be the number of the edges of $F_j$. Let $S$ denote the sum of all interior angles of all cells of $\M_i$ in $R \BB^2_i$. Then, since $\M_i$ is normal, on one hand
\[
S = 2 \pi N_i(R) + \mathcal{O}(R),
\]
and on the other hand it is
\[
S=\sum_{j=1}^s (e_j-2) \pi + \mathcal{O}(R).
\]
Since $\area(F_j)=a_i$ for all $j$, $s = \frac{R^2 \pi}{a_i} + \mathcal{O}(R)$. This yields that
\[
N_i(R) = \frac{R^2 \pi}{2 a_i} \left( \frac{\sum_{j=1}^s e_j}{s} - 2 \right) + \mathcal{O}(R).
\]
Let $e_i(R) = \frac{1}{s} \sum_{j=1}^s e_j$ denote the average number of edges of the cells of $\M_i$ contained in $R \BB_i^2$. Then
\begin{equation}\label{eq:Ni}
N_i(R) = \frac{R^2 \pi }{2 a_i} \left( e_i(R) - 2 \right) + \mathcal{O}(R).
\end{equation}
Here we recall the well-known fact that for any normal, convex mosaic $\M_i$ in the plane, we have $3 \leq \liminf_{R \to \infty} e_i(R) \leq \limsup_{R \to \infty} e_i(R) \leq 6$.

Next, we examine the length of the $1$-skeleton of $\M_i$. We have
\[
L(\skel_1(\M_i) \cap R \BB^2_i) = \frac{1}{2} \sum_{j=1}^s \perim(F_j) + \mathcal{O}(R).
\]
By the Discrete Isoperimetric Inequality, $\perim(F_j) \geq 2 \sqrt{ e_j \tan \frac{\pi}{e_j} a_i}$; here the right-hand side is the perimeter of the regular $e_j$-gon of area $a_i$. Thus, by the convexity of the function $x \mapsto \sqrt{ x \tan \frac{\pi}{x}}$ on the interval $[3,6]$, we have that
\begin{equation}\label{eq:skeli}
L(\skel_1(\M_i) \cap R \BB^2_i) \geq \frac{R^2 \pi}{a_i} \cdot \sqrt{ e_i(R) \tan \frac{\pi}{e_i(R)} a_i} + \mathcal{O}(R).
\end{equation}

Now we return to the investigation of $\M$. Combining the formulas in (\ref{eq:Ni}) and (\ref{eq:skeli}) with the description of the edges of $\M$, we obtain that   
\[
L\left(\skel_1(\M) \cap \left( R \cdot \bigtimes_{j=1}^k \mathbb{B}_j^2 \right) \right) = \sum_{i=1}^k L(\skel_1(\M_i) \cap R \BB^2_i) \prod_{j=1, j \neq i}^k N_j(R) + \mathcal{O}(R^{2k-1})\geq
\]
\[
\geq \sum_{i=1}^k \frac{R^2 \pi}{a_i} \cdot \sqrt{ e_i(R) \tan \frac{\pi}{e_i(R)} a_i} \prod_{j=1, j \neq i}^k  \frac{R^2 \pi}{2 a_j} \left( e_j(R) - 2 \right) + \mathcal{O}(R^{2k-1}).
\]
Recall that $\prod_{j=1}^k a_j = 1$ by our conditions. Thus,
\[
L\left(\skel_1(\M) \cap \left( R \bigtimes_{j=1}^k \mathbb{B}_j^2 \right) \right) \geq \frac{R^{2k} \pi^k}{2^{k-1}} \sum_{i=1}^k \sqrt{ e_i(R) \tan \frac{\pi}{e_i(R)} a_i} \prod_{j=1, j \neq i}^k \left( e_j(R) - 2 \right) + \mathcal{O}(R^{2k-1}),
\]
implying, by the compactness of the interval [3,6], that
\begin{equation}\label{eq:evenfinalestimate}
\underline{\rho}_1(\M) \geq \frac{1}{2^{k-1}} \sum_{i=1}^k \sqrt{ \hat{e}_i \tan \frac{\pi}{\hat{e}_i} a_i} \prod_{j=1, j \neq i}^k \left( \hat{e}_j - 2 \right)
\end{equation}
for some values $\hat{e}_i \in [3,6]$.
By the AM-GM inequality, from this it follows that
\[
\underline{\rho}_1(\M) \geq \frac{k}{2^{k-1}} \left( \prod_{i=1}^k \sqrt{\hat{e}_i  \tan \frac{\pi}{\hat{e}_i}} \left(\hat{e}_i - 2 \right)^{k-1}   \right)^{1/k}.
\]
Since $\hat{e}_i \in [3,6]$ for all values of $i$, it follows that
\[
\underline{\rho}_1(\M) \geq \frac{k}{2^{k-1}} \inf \left\{ \sqrt{x  \tan \frac{\pi}{x}} \left( x - 2 \right)^{k-1} : x \in [3,6] \right\}.
\]
Now, an elementary computation shows that the function $x \mapsto \sqrt{x  \tan \frac{\pi}{x}}$ strictly decreases whereas the function $x \mapsto (x-2) \sqrt{x  \tan \frac{\pi}{x}}$ strictly increases on $x \in [3,6]$. Thus, $\underline{\rho}_1(\M) \geq \sqrt{2 \sqrt{3}}$ if $k=1$, and $\underline{\rho}_1(\M) \geq \frac{k\sqrt{3 \sqrt{3}}}{2^{k-1}}$ if $k \geq 2$.

\emph{Case 2}: $n=2k+1$ is odd.\\

Let $\M$ be the sum of the planar mosaics $\M_1, \M_2, \ldots, \M_k$, and the $1$-dimensional mosaic $\mathcal{N}$. Similarly like in Case 1, for every $i$, $\M_i$ contains equal area cells, and $\mathcal{N}$ consists of equal length segments. Let the area of the cells of $\M_i$ be denoted by $a_i$, and let $l$ denote the length of the segments of $\mathcal{N}$. Clearly, all the components of $\M$ are normal mosaics. We assume that each $\M_i$ is an edge-to-edge mosaic.

Like in Case 1, we have
\begin{equation}
\underline{\rho}_1 (\M) = \liminf_{R \to \infty} \frac{L\left(\skel_1(\M) \cap \left( R \bigtimes_{j=1}^k \mathbb{B}_j^2 \times [-R,R]\right) \right)}{\vol\left( R \bigtimes_{j=1}^k \mathbb{B}_j^2 \times [-R,R] \right)}.
\end{equation}
Based on this formula and using the notation and approach of Case 1, we obtain that
\begin{multline*}
L\left(\skel_1(\M) \cap \left( R \bigtimes_{j=1}^k \mathbb{B}_j^2 \right) \right) = \sum_{i=1}^k L(\skel_1(\M_i) \cap R \BB^2_i) \prod_{j=1, j \neq i}^k N_j(R) \cdot N(R)+ 
2R \prod_{j=1}^k N_j(R) + \mathcal{O}(R^{2k}),
\end{multline*}
where $N(R)$ denotes the number of vertices of $\mathcal{N}$ in the interval $[-R,R]$. Hence,
\[
L\left(\skel_1(\M) \cap \left( R \bigtimes_{j=1}^k \mathbb{B}_j^2 \right) \right) \geq \sum_{i=1}^k \frac{R^2 \pi}{a_i} \cdot \sqrt{ e_i(R) \tan \frac{\pi}{e_i(R)} a_i} \prod_{j=1, j \neq i}^k  \frac{R^2 \pi}{2 a_j} \left( e_j(R) - 2 \right) \cdot \frac{2R}{l} +
\]
\[
+ 2R \prod_{i=1}^k \frac{R^2 \pi}{2 a_j} \left( e_j(R) - 2 \right) + \mathcal{O}(R^{2k}) = 
\]
\[
= \frac{R^{2k+1} \pi^k}{2^{k-2}} \sum_{i=1}^k \sqrt{e_i(R) \tan \frac{\pi}{e_i(R)} a_i} \prod_{j=1,j\neq i}^k \left( e_j(R)-2\right) +
\frac{R^{2k+1} \pi^k l}{2^{k-1}} \prod_{j=1}^k \left( e_j(R)-2 \right) + \mathcal{O}(R^{2k})
\]
By the compactness of the interval $[3,6]$, this yields that
\[
\underline{\rho}_1(\M) \geq \frac{1}{2^{k-1}} \sum_{i=1}^k \sqrt{\hat{e}_i \tan \frac{\pi}{\hat{e}_i} a_i} \prod_{j=1,j\neq i}^k \left( \hat{e}_j-2\right) +
\frac{l}{2^{k}} \prod_{j=1}^k \left( \hat{e}_j-2 \right)
\]
for some values $\hat{e}_i \in [3,6]$. By the AM-GM Inequality, and using the fact that $\prod_{i=1}^k a_i = \frac{1}{l}$, from this it follows that
\[
\underline{\rho}_1(\M) \geq \frac{k}{2^{k-1}} \left( \prod_{i=1}^k \sqrt{\hat{e}_i \tan \frac{\pi}{\hat{e}_i}} \left( \hat{e}_i-2\right)^{k-1}\right)^{1/k} \cdot \frac{1}{\sqrt{l}}+
\frac{1}{2^k} \prod_{i=1}^k \left( \hat{e}_i-2\right) \cdot l.
\]
Now we show that under the condition that $l$ and $\prod_{i=1}^k \left( \hat{e}_i-2\right)$ are fixed, the function $\prod_{i=1}^k \hat{e}_i \tan \frac{\pi}{\hat{e}_i}$ is minimal if all the $\hat{e}_i$ are equal. Indeed, set $x_i = \ln (\hat{e_i}-2) \in [0, \ln 4]$ for $i=1,2,\ldots,k$ and let $g(x) = \ln \left( (e^x+2) \tan \frac{\pi}{e^x+2}\right)$. With this notation, we need to show that if $\sum_{i=1}^k x_i$ is fixed, then $\sum_{i=1}^k g(x_i)$ is minimal if all the $x_i$ are equal. But a simple computation shows that $g''(x) >0 $ for all $x \in [0, \ln 4]$, and hence the statement follows from the convexity of the function $g$.

Let $t= 2+\left( \prod_{i=1}^k \left( \hat{e}_i-2\right) \right)^{1/k} \in [3,6]$. Then, by our previous consideration, we have
\[
\underline{\rho}_1(\M) \geq \frac{k}{2^{k-1}} \sqrt{t \tan \frac{\pi}{t}} (t-2)^{k-1} \cdot \frac{1}{\sqrt{l}} + \frac{1}{2^k} (t-2)^k \cdot l
\]
for some $l \in (0,1)$ and $t \in [3,6]$. Differentiating (or equivalently, by the AM-GM inequality), we have that for any fixed value of $t$, the right-hand side is minimal if $l=\frac{\left( k \sqrt{t \tan \frac{\pi}{t}} \right)^{2/3}}{(t-2)^{2/3}}$. Substituting it into the above inequality, it follows that
\[
\underline{\rho}_1(\M) \geq \frac{3 k^{2/3}}{2^k} \left( t \tan \frac{\pi}{t} (t-2)^{3k-2} \right)^{1/3}
\]
for some $t \in [3,6]$. But, as we have already concluded in Case 1, the function $t \mapsto t \tan \frac{\pi}{t} (t-2)$ is strictly increasing on the interval $[3,6]$, implying that the right-hand side is minimal for $t=3$.

\section{Preliminaries for the proofs of Theorems~\ref{thm:1} and \ref{thm:totaledgelength}}\label{sec:prelim}

In this section we present a description and parametrization of $3$-dimensional parallelohedra introduced in \cite{Langi}.

It is shown in \cite{Minkowski, Venkov} that an $n$-dimensional convex polytope $P$ is a parallelohedron if and only if
\begin{itemize}
\item[(i)] $P$ and all its facets are centrally symmetric, and
\item[(ii)] the projection of $P$ along any of its $(n-2)$-dimensional faces is a parallelogram or a centrally symmetric hexagon.
\end{itemize}
It is worth noting that any affine image of a parallelohedron is also a parallelohedron, and by \cite{Dolbilin}, for any $n$-dimensional parallelohedron $P$ there is a \emph{face-to-face} tiling of $\Re^n$ with translates of $P$ as cells.

The combinatorial classes of $3$-dimensional parallelohedra were described a long time ago \cite{Fedorov}. More specifically, any 
$3$-dimensional parallelohedron is combinatorially isomorphic to one of the following (see Figure~\ref{fig:types}):
\begin{enumerate}
\item[(1)] a cube,
\item[(2)] a hexagonal prism,
\item[(3)] Kepler's rhombic dodecahedron (which we call a regular rhombic dodecahedron),
\item[(4)] an elongated rhombic dodecahedron,
\item[(5)] a regular truncated octahedron.
\end{enumerate}
We call a parallelohedron combinatorially isomorphic to the polyhedron in (i) a \emph{type (i) parallelohedron}.

\begin{figure}[ht]
\begin{center}
\includegraphics[width=0.6\textwidth]{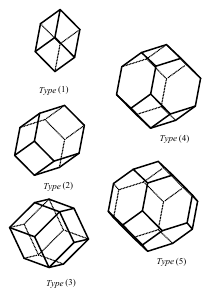}
\caption{Combinatorial types of $3$-dimensional parallelohedra. The type (5) parallelohedron is the regular truncated octahedron generated by the six segments connecting the midpoints of opposite edges of a cube. The type (3) polyhedron is the regular rhombic dodecahedron generated by the four diagonals of the same cube. The other parallelohedra in the picture are obtained by removing some generating segments from the type (5) parallelohedron.}
\label{fig:types}
\end{center}
\end{figure}

Every $3$-dimensional parallelohedron is a zonotope; in particular, a type (1)-(5) parallelohedron is the Minkowski sum of $3,4,4,5,6$ segments, respectively. A typical parallelohedron is of type (5), as every other parallelohedron can be obtained by removing some of the generating vectors of a type (5) parallelohedron, or setting their lengths to zero.  

By well-known properties of zonotopes (see e.g. \cite{JL, Langi, Schneider}), every edge of a $3$-dimensional parallelohedron $P$ is a translate of one of the segments generating $P$. Furthermore, for any generating segment $S$, the faces $F_i$ of $P$ that contain a translate of $S$ can be written in a cyclic sequence $F_1, F_2, \ldots, F_s = F_0$ such that for each $i$, $F_i \cap F_{i+1}$ is a translate of $S$; the union of these faces is called the \emph{zone} of $S$.
By property (ii) of parallelohedra, for any such zone the number of elements is $4$ or $6$, implying that for any generating segment $S$ of $P$ there are $4$ or $6$ translates of $S$ among the edges of $P$. In these cases we say that $S$ generates a \emph{$4$-belt} or a \emph{$6$-belt}, respectively.
For $i=1,2,3,4,5$, a type (i) parallelohedron contains $3, 3, 0, 1, 0$ $4$-belts, and $0,1,4,4,6$ $6$-belts, respectively.

If $P=\sum_{i=1}^k S_i$ is a $3$-dimensional parallelohedron generated by the segments $S_i$, and $m=(\alpha_6,\alpha_4)$ has positive coordinates, we set
\[
w_i = \left\{
\begin{array}{l}
\alpha_6 \hbox{ if } S_i \hbox{ generates a } 6- \hbox{belt;}\\
\alpha_4 \hbox{ if } S_i \hbox{ generates a } 4- \hbox{belt.}
\end{array}
\right.
\]
Furthermore, we define the function $w_{m}(P) = \sum_{i=1}^k w_i L(S_i)$, where $L(S_i)$ denotes the length of $S_i$.
We prove Theorem~\ref{thm:2}, in which $P_{cu}$ denotes a unit cube, $P_{oc}$ denotes a unit volume regular truncated octahedron, and $P_{pr}$ denotes a regular hexagon-based right prism, where the edge length of the hexagon is $\frac{2^{2/3} \alpha_6^{1/3}}{3^{5/6} \alpha_4^{1/3}}$, and the length of the lateral edges of $P_{pr}$ is $\frac{3^{1/6} \alpha_4^{2/3}}{2^{1/3} \alpha_6^{2/3}}$. We note that the volume of $P_{pr}$ is also one.

\begin{thm}\label{thm:2}
Let $m=(\alpha_6,\alpha_4)$ with $\alpha_4, \alpha_6 > 0$. Then $w_m(\cdot)$ attains its minimum on $\PP_{tr}$. Furthermore, if $P \in \PP_{tr}$ minimizes $w_m(\cdot)$ on $\PP_{tr}$, then the following holds.
\begin{itemize}
\item[(a)] If $0 < \alpha_4 < \frac{\sqrt{3}}{2} \alpha_6$, then $P$ is congruent to $P_{cu}$, and $w_m(P) = 3 \alpha_4$.
\item[(b)] If $\alpha_{4} = \frac{\sqrt{3}}{2} \alpha_6 ( \approx 0.866025 \alpha_6)$, then $P$ is congruent to $P_{cu}$ or $P_{pr}$.
\item[(c)] If $\frac{\sqrt{3}}{2} \alpha_6 < \alpha_4 < \sqrt[4]{\frac{2}{3}} \alpha_6$, then $P$ is congruent to $P_{pr}$, and $w_m(P)= \frac{3^{7/6} \alpha_4^{2/3} \alpha_6^{1/3}}{2^{1/3}}$.
\item[(d)] If $\alpha_4 = \sqrt[4]{\frac{2}{3}} \alpha_6 (\approx 0.903202 \alpha_6)$, then $P$ is congruent to $P_{pr}$ or $P_{oc}$.
\item[(e)] If $\sqrt[4]{\frac{2}{3}} \alpha_6 < \alpha_4$, then $P$ is congruent to $P_{oc}$, and $w_m(P)=\frac{3}{2^{1/6}} \alpha_6$.
\end{itemize}
\end{thm}

We note that if $m=\left( \frac{1}{2}, \frac{1}{2} \right)$, then $w_m(P)$ coincides with the mean width of $P$ for any $P \in \PP_{tr}$ \cite[Remark 3]{Langi}, and thus, Theorem~\ref{thm:2} is a generalization of the main result, Theorem 1, of \cite{Langi}, stating that among $3$-dimensional unit volume parallelohedra, regular truncated octahedra have minimal mean width (for the definition of mean width, see e.g. \cite{Gardner, Schneider}).

To prove Theorem~\ref{thm:2} we adapt the representation of $3$-dimensional parallelohedra introduced in \cite{Langi}.
More specifically, if $P$ is a $3$-dimensional parallelohedron, then, up to translation, there are vectors $v_1, v_2,v_3, v_4 \in \Re^3$ such that
\begin{equation}\label{eq:representation}
P = \sum_{1 \leq i < j \leq 4} [o, \beta_{ij}(v_i \times v_j)]
\end{equation}
for some non-negative real numbers $\beta_{ij}$. Here, if $P$ is a type (5) parallelohedron, the vectors $v_i$ are normal vectors of the pairs of hexagon faces.
As it was shown in \cite{Langi}, we may assume that any three of the $v_i$ are linearly independent, and satisfy $v_1+v_2+v_3+v_4=o$. We may also assume that $V_{123}=1$, where $V_{ijk}$ denotes the value of the determinant with columns $v_i, v_j, v_k$. Then, setting $T= \conv \{ v_1, v_2, v_3, v_4 \}$, $T$ is \emph{centered}, i.e. its center of mass is $o$, implying that for any pairwise different indices $i,j,k$, we have $|V_{ijk}| = 1$. The sign of $V_{ijk}$ can be determined from our choice that $V_{123}=1$, the properties of determinants, and from the identity that for any $\{i,j,s,t\} = \{ 1,2,3,4\}$, we have $V_{ijs}=-V_{ijt}$.

Every $3$-dimensional parallelohedron can be obtained from a type (5) parallelohedron by setting some of the $\beta_{ij}$s to be zero. In particular, $P$ is of type (5), or (4), or (3), or (2), or (1) if and only if
\begin{itemize}
\item $\beta_{ij} > 0$ for all $i \neq j$, or
\item exactly one of the $\beta_{ij}$s is zero, or
\item exactly two of the $\beta_{ij}$s is zero: $\beta_{i_1j_1} = \beta_{i_2,j_2} = 0$, and they satisfy $\{ i_1, j_1 \} \cap \{ i_2, j_2 \} = \emptyset$, or
\item exactly two of the $\beta_{ij}$s is zero: $\beta_{i_1j_1} = \beta_{i_2,j_2} = 0$, and they satisfy $\{ i_1, j_1 \} \cap \{ i_2, j_2 \} \neq \emptyset$, or
\item exactly three of the $\beta_{ij}$s are zero, and there is no $s \in \{1,2,3,4\}$ such that the indices of all nonzero $\beta_{ij}$s contain $s$, respectively.
\end{itemize}
In the remaining cases $P$ is planar.

The volume of a parallelohedron $P$ in the above representation is independent of the $v_i$. More specifically,
\begin{equation}\label{eq:vol_rep}
\vol_3(P)= f(\beta_{12},\beta_{13},\beta_{14},\beta_{23},\beta_{24},\beta_{34}),
\end{equation}
where the function $f : \Re^6 \to \Re$ is defined as
\begin{multline}\label{eq:volume}
f(\tau_{12},\tau_{13},\tau_{14},\tau_{23},\tau_{24},\tau_{34}) = \tau_{12} \tau_{13} \tau_{23} + \tau_{12}\tau_{14}\tau_{24} + \tau_{13}\tau_{14}\tau_{34} + \tau_{23} \tau_{24} \tau_{34} +\\
+ (\tau_{12}+\tau_{34})(\tau_{13}\tau_{24}+\tau_{14}\tau_{23})+ (\tau_{13}+\tau_{24})(\tau_{12}\tau_{34}+\tau_{14}\tau_{23}) + (\tau_{14}+\tau_{23})(\tau_{12}\tau_{34}+\tau_{13}\tau_{24}).
\end{multline}
Note that as $f$ has only $16 < 20 = \binom{6}{3}$ members, it is not the third elementary symmetric function on six variables.

Lemma~\ref{lem:tetrahedron_identities} is proved as \cite[Lemma 4]{Langi}.

\begin{lem}\label{lem:tetrahedron_identities}
Let $T=\conv \{ p_1,p_2,p_3,p_4\}$ be an arbitrary centered tetrahedron with volume $V>0$ where the vertices are labelled in such a way that the determinant with columns $p_1,p_2,p_3$ is positive.
For any $\{i,j,s,t\} = \{1,2,3,4\}$, let $\gamma_{ij}=- \langle p_s, p_t \rangle$
and $\zeta_{ij}= \gamma_{ij} | p_i \times p_j|^2$. Then
\begin{enumerate}
\item $f(\gamma_{12},\gamma_{13},\gamma_{14},\gamma_{23},\gamma_{24},\gamma_{34}) = \frac{9}{4}V^2$,
\item $\sum_{1 \leq  i < j \leq 4} \zeta_{ij} = \frac{27}{4}V^2$,
\end{enumerate}
where $f$ is the function defined in (\ref{eq:volume}).
\end{lem}
 

To prove our result, we need the following variant of \cite[Lemma 5]{Langi}.  Here, for any $C > 0$, we set
$\X_C = \left\{(\tau_{12}, \dots, \tau_{34}) \in \Re^6_{\geq 0}: \, \sum_{1 \leq i < j \leq 4} \tau_{ij} = C\right\}$.

\begin{lem}\label{lem:computations2}
Let $C > 0$ and $\lambda \geq 1$. Then for each $(\tau_{12},\tau_{13},\tau_{14},\tau_{23},\tau_{24},0)\in \mathcal{X}_C$ we have 
\begin{equation}\label{ineq:g}
 f(\lambda \tau_{12},\tau_{13},\tau_{14},\tau_{23},\tau_{24},0) \leq \frac{16C^3 \lambda^3}{27(4\lambda - 1)^2},
\end{equation}
with equality if and only if
\[
\tau_{12} = \frac{4 \lambda - 3}{2\lambda}\tau_{13} = \ldots = \frac{4 \lambda - 3}{2\lambda} \tau_{24} = \frac{2 \lambda}{12 \lambda - 3}C.
\]
\end{lem}

\begin{proof}
Without loss of generality, we assume that $C=1$ and for simplicity, in the proof we denote the operator $\frac{\partial}{\partial \tau_{ij}}$ by $\partial_{ij}$.
Given that $\mathcal{X}_C$ is compact, the function $f$ attains its maximum on it at some point $(\tau^*,0) \in \mathcal{X}_C$ with $\tau^*=(\tau^*_{12}, \ldots, \tau_{24}^*)$.

\emph{Case 1}: all coordinates of $\tau^*$ are positive.\\
Then, by the Lagrange multiplier method, the gradient of $f$ at $\tau^*$ is parallel to the gradient of the function $\sum_{i < j, ij \neq 34} \tau_{ij}-1$ defining the condition, thus, all five partial derivatives of $f$ at $\tau^*$ are equal. 
From the equations $(\partial_{13} f)(\tau^*) = (\partial_{14} f)(\tau^*)$ and $(\partial_{23} f)(\tau^*) = (\partial_{24} f)(\tau^*)$ and by the assumption that all coordinates of $\tau^*$ are positive, we obtain that $\tau_{13}^* = \tau_{14}^*$ and $\tau_{23}^* = \tau_{24}^*$. Eliminating $\tau_{14}^*$ and $\tau_{24}^*$ from the partial derivatives, the equation $(\partial_{13} f)(\tau^*) = (\partial_{23} f)(\tau^*)$, and the fact that all $\tau_{ij}^*$ as well as $\lambda$ are positive, yields that $\tau_{13}^*=\tau_{23}^*$. Thus, we need to find the maximum of $f(\lambda \tau_{12},\tau_{13},\ldots,\tau_{13},0)$ under the condition that $\tau_{12} + 4 \tau_{13} = 1$. From this, by simple calculus methods, we obtain that $f(\lambda \tau_{12},\tau_{13},\ldots,\tau_{13},0) \leq \frac{16 \lambda^3}{27(4\lambda - 1)^2}$, with equality if and only if $\tau_{12} = \frac{4 \lambda - 3}{2\lambda}\tau_{13} = \frac{2 \lambda}{12 \lambda - 3}$.

\emph{Case 2}: $\tau_{12}^* = 0$ and all other $\tau_{ij}^*$ are positive.\\
Note that $f(0,\tau_{13},\ldots,\tau_{24},0) = (\tau_{13}+\tau_{24})\tau_{14}\tau_{23}+(\tau_{14}+\tau_{23})\tau_{13}\tau_{24}$. In this case, by the same argument as in Case 1, we obtain that $\tau_{13}^*=\ldots=\tau_{24}^*=\frac{1}{4}$, and $f(\tau^*)=\frac{1}{16}$. 

\emph{Case 3}: $\tau_{ij}^* = 0$ for some $ij \neq 12$, and all other $\tau_{st}^*$ are positive.\\
Without loss of generality, we may assume that $\tau_{24}^* = 0$. By a method similar to the one in Case 1 and using the condition that $\lambda \geq 1$, we obtain that $\tau_{12}^* = \frac{2(2\lambda-1)}{3(4\lambda-1)}$, $\tau_{13}^* = \tau_{14}^* = \frac{2\lambda}{3(4\lambda-1)}$, and $\tau_{23}^* = 0$, and $f(\tau^*)=\frac{4\lambda^2}{27(4\lambda-1)}$.

\emph{Case 4}, at least two $\tau_{ij}^*$ are zero.\\
We may assume that exactly two $\tau_{ij}^*$ are zero, and that there is no $s \in \{1,2,3,4\}$ such that the indices of all nonzero $\tau_{ij}^*$s contain $s$, as otherwise $f(\tau^*,0)=0$. Under this condition, depending on which $\tau_{ij}^*$ are not zero, we have that $f(\tau^*)=\frac{\lambda}{27}$ or $f(\tau^*)=\frac{1}{27}$, with equality if and only if all nonzero $\tau_{ij}^*$ are equal to $\frac{1}{3}$.

As a summary, we need to find
\[
\max \left \{ \frac{16 \lambda^3}{27(4\lambda - 1)^2}, \frac{1}{16}, \frac{4\lambda^2}{27(4\lambda-1)}, \frac{\lambda}{27}, \frac{1}{27} \right\}.
\]
Then an elementary computation shows that under the condition $\lambda \geq 1$, the maximum is $\frac{16 \lambda^3}{27(4\lambda - 1)^2}$, and the assertion follows.
\end{proof}

\section{The proof of Theorems~\ref{thm:1} and \ref{thm:totaledgelength} using Theorem~\ref{thm:2}}\label{sec:thm1}

Setting $\alpha_6=6$ and $\alpha_4=4$ in Theorem~\ref{thm:2}, we readily obtain Theorem~\ref{thm:totaledgelength}.

Theorem~\ref{thm:1} follows from Theorem~\ref{thm:2} by Lemma~\ref{lem:weights}.

\begin{lem}\label{lem:weights}
Let $P$ be a unit volume parallelohedron in $\Re^3$, and let $M$ be a tiling of $\Re^3$ with translates of $P$.
Set $m=(2,1)$. Then $\underline{\rho}_1(\M) \geq w_m(P)$, with equality if $\M$ is face-to-face.
\end{lem}

\begin{proof}
First, we assume that $M$ is face-to-face and that $P$ is a cell of $M$.
Let us choose an edge $E$ of $P$. Let $\mathcal{E}$ contain the edges $E'$ of $\M$ for which there is a sequence of cells $P_1, \ldots, P_k$ such that
\begin{enumerate}
\item[(i)] for any $i=1,2,\ldots,k-1$, $P_i \cap P_{i+1}$ is a face containing an edge that is a translate of $E$;
\item[(ii)] $E$ is an edge of $P_1$ and $E'$ is an edge of $P_k$.
\end{enumerate}
Furthermore, let $\mathcal{P}$ be the family of cells containing edges from $\mathcal{E}$.

Note that the projection of the members of $\mathcal{P}$ in the direction of $E$ onto a plane $H$ is a translative, convex, edge-to-edge tiling of $H$. By the second property of parallelohedra in the list in the beginning of Section~\ref{sec:prelim}, the projections of the elements of $\mathcal{P}$ are either parallelograms, in which every vertex belongs to exactly four parallelograms, or they are centrally symmetric hexagons, in which case every vertex belongs to three hexagons. Thus, a cell contains $4$ or $6$ translates of any given edge, and in the first case every edge belongs to exactly $4$ cells, and in the second case every edge belongs to exactly $3$ cells.

Now, we assign a weight $w_E$ to each edge $E$ of $M$; we set $w_E = \frac{1}{k}$, where $k$ is the number of cells $E$ belongs to.
Let $N_R$ denote the number of cells of $\M$, and $\mathcal{E}_R$ denote the family of the edges of the cells contained in $R \BB^3$.
As $N_R = \vol_3(R \BB^3)+ \mathcal{O}(R^2)$, we have
\[
L(\skel_1 (\M) \cap R \BB^3) = \sum_{E \in \mathcal{E}_R} L(E)  + \mathcal{O}(R^2) = N_r \sum_{i=1}^6 w_i |p_i| + \mathcal{O}(R^2),
\]
implying Lemma~\ref{lem:weights} for face-to-face mosaics.

Finally, if $\M$ is a not necessarily face-to-face tiling by translates of $P$, then we may divide each edge into finitely many pieces in such a way that the relative interior points on each piece belong to the same cells. It is easy to see that any part of an edge in a $6$-belt belongs to exactly $3$ cells, and a part of an edge in a $4$-belt belongs to either $4$ or $2$ cells. Thus, repeating the procedure, the weighted sum of the edge lengths in $R \BB^3$, defined as in the previous case, yields a value for $\underline{\rho}_1(M)$ asymptotically not less than that for a face-to-face mosaic generated by $P$.
\end{proof}

\section{The proof of Theorem~\ref{thm:2}}\label{sec:thm2}

First, for each value of $i$ and $m=(\alpha_4,\alpha_6)$, we find the infimum of $w_m(P)$ of the type (i) parallelohedra for each $i$ separately, which we denote by $w_m^i$, then compare these values.

\noindent
\textbf{Case 1}: $P$ is a type (1) parallelohedron.\\
Then $P$ is a parallelopiped, and it is easy to see (e.g. by the AM-GM inequality) that $w_m(P)$ is minimal if and only if $P$ is a unit cube, and in this case $w_m(P) = 3 \alpha_4$. Thus, $w_m^1 = 3 \alpha_4$.

\noindent
\textbf{Case 2}: $P$ is a type (2) parallelohedron.\\
Then $P$ is a centrally symmetric hexagon-based prism. By Cavalieri's principle, we can assume that $P$ is a right prism, and by the Discrete Isoperimetric Inequality, we may assume that its base is a regular hexagon. Let the edge length of the base be $a$, and the length of the lateral edges be $b$. Then $\vol_3(P) = \frac{3\sqrt{3}}{2}a^2 b$, and $w_m(P) = 3 \alpha_4 a + \alpha_6 b$. An elementary computation shows that if $\frac{3\sqrt{3}}{2}a^2 b = 1$, then $w_m(P) \geq \frac{3^{7/6} \alpha_4^{2/3} \alpha_6^{1/3}}{2^{1/3}}$, with equality if and only if $a= \frac{2^{2/3} \alpha_6^{1/3}}{3^{5/6} \alpha_4^{1/3}}$ and $b=\frac{3^{1/6} \alpha_4^{2/3}}{2^{1/3} \alpha_6^{2/3}}$. Hence, $w_m^2 = \frac{3^{7/6} \alpha_4^{2/3} \alpha_6^{1/3}}{2^{1/3}}$.

\noindent
\textbf{Case 3}: $P$ is a type (3) parallelohedron.\\
In this case $P$ is generated by $4$ segments, and each segment generates a $6$-belt. Thus, $w_m(P) = \frac{\alpha_6}{2} w(P)$, where $w(P)$ denotes the mean width of $P$ (see the sentence after Theorem~\ref{thm:2}). This means that we may apply the proof of \cite[Theorem 1]{Langi} (see also Table 1 of \cite{Langi}), and obtain that $w_m(P) \geq 3^{1/2} 2^{2/3} \alpha_6$, with equality if and only if $P$ is a regular rhombic dodecahedron, implying, in particular, that $w_m^3 = 3^{1/2} 2^{2/3} \alpha_6$.

\noindent
\textbf{Case 4}: $P$ is a type (4) parallelohedron.

First, note that if $\alpha_4 > \alpha_6$, then, with the notation $m'=(\alpha_6,\alpha_6)$, $w_m^4 \geq w_{m'}^4 > w_{m'}^5 = w_m^5$, where the last inequality follows from \cite[Theorem 1]{Langi}. In the following we assume that $\alpha_4 \leq \alpha_6$. 

To prove Theorem~\ref{thm:2} in this case, we use the notation introduced in Section~\ref{sec:prelim}: we represent $P$ in the form
\[
P= \sum_{1 \leq i < j \leq 4} [o,\beta_{ij} v_i \times v_j]
\]
for some  $v_1, v_2, v_3, v_4 \in \Re^3$ satisfying $\sum_{i=1}^4 v_i = o$, and for some $\beta_{ij} \geq 0$, with exactly one $\beta_{ij}$ equal to zero. By our assumptions, $|V_{ijk}| = 1$ for any $\{i,j,k \} \subset \{ 1,2,3,4\}$, where $V_{ijk}$ denotes the determinant with $v_i,v_j,v_k$ as columns, and we assumed that $V_{123}=1$, which implies, in particular, that $V_{124}=-1$.
Then $T=\conv \{ v_1, v_2, v_3, v_4 \}$ is a centered tetrahedron, with volume $\vol_3(T)=\frac{2}{3}$.

Without loss of generality, we set $\beta_{34} = 0$. Then, among the edges of $P$, the translates of $\beta_{12} v_1 \times v_2$ belong to a $4$-belt, and the translates of all other edges belong to $6$-belts. Thus, we have
\[
w_m(P) = \alpha_4 \beta_{12} | v_1 \times v_2| + \alpha_6 \sum_{i \in \{1,2 \}, j \in \{ 3,4 \} }\beta_{ij} |v_i \times v_j|, \hbox{ and}
\]
\[
\vol_3(P) = f(\beta_{12}, \beta_{13},\ldots, \beta_{24}, 0).
\]

We follow the two steps of the proof of \cite[Theorem 1]{Langi}. Our first step is to find the minimum of $w_m(P)$ in the affine class of $P$. For this, we use Lemma~\ref{lem:isotropic}, proved by Petty \cite{Petty} (see also (3.4) in \cite{Giannopoulos} for arbitrary convex bodies).

\begin{lem}\label{lem:isotropic}
Let $P \subset \Re^n$ be a convex polytope with outer unit facet normals $u_1, \ldots, u_k$. Let $F_i$ denote the $(n-1)$-dimensional volume of the $i$th facet of $P$. Then, up to congruence, there is a unique volume preserving affine transformation $A$ such that the surface area $\surf(A(P))$ is maximal in the affine class of $P$. Furthermore, $P$ satisfies this property if and only if its surface area measure is isotropic, that is, if
\begin{equation}\label{eq:isotropic}
\sum_{i=1}^k \frac{n F_i}{\surf(P)} u_i \otimes u_i = \Id
\end{equation}
where $\Id$ denotes the identity matrix.
\end{lem}

Any convex polytope satisfying the conditions in (\ref{eq:isotropic}) is said to be in \emph{surface isotropic position}.
We note that the volume of the projection body of any convex polyhedron is invariant under volume preserving linear transformations (cf. \cite{Petty2}).
On the other hand, from Cauchy's projection formula and the additivity of the support function (see \cite{Gardner}) it follows that the projection body of the polytope in Lemma~\ref{lem:isotropic} is the zonotope $\sum_{i_1}^k [o,F_i u_i]$ (for an elementary explanation of this formula, see \cite[Theorem 1.1]{BGK}). Note that the solution to Minkowski's problem \cite[Theorem 8.2.2]{Schneider} states that some unit vectors $u_1, \ldots, u_k \in \Re^n$ and positive numbers $F_1, \ldots, F_k$ are the outer unit normals and volumes of the facets of a convex polytope if and only if the $u_i$s span $\Re^n$, and $\sum_{i=1}^k F_i u_i = o$. This yields that there is an $o$-symmetric convex polyhedron $Q$ whose faces have outer unit normals $\pm \frac{v_1 \times v_2}{|v_1 \times v_2|}$, and $\pm \frac{v_i \times v_j}{|v_i \times v_j|}$ with $ij=13,23,14,24$, and these faces have area $\frac{1}{2} \alpha_4 \beta_{12} | v_1 \times v_2|$ and $\frac{1}{2} \alpha_6 \beta_{ij} |v_i \times v_j|$ with $ij=13,14,23,24$, respectively.
Since $u \otimes u = (-u) \otimes (-u)$ for any $u \in \Re^n$, we may apply Lemma~\ref{lem:isotropic} for $Q$, and obtain that
\begin{equation}\label{eq:isotropic_w}
\frac{\alpha_4 \beta_{12}}{|v_1 \times v_2|} (v_1 \times v_2) \otimes (v_1 \times v_2) + \alpha_6 \sum_{i \in \{1,2 \}, j \in \{ 3,4 \} } \frac{2\beta_{ij}}{|v_i \times v_j|} (v_i \times v_j) \otimes (v_i \times v_j) = \frac{w_m(P)}{3} \Id,
\end{equation}
which in the following we assume to hold for $P$.

Recall that $x \otimes y = x y^T$ and $\langle x,y \rangle = x^T y$ for any column vectors $x,y \in \Re^n$.
Hence, multiplying both sides in (\ref{eq:isotropic_w}) by $v_3^T$ from the left and $v_4$ from the right, it follows that
\[
\langle v_3, v_4 \rangle = \frac{3 \alpha_4 \beta_{12}}{w_m(P) |v_1 \times v_2|} V_{123} V_{124}.
\]
Since $V_{123} = - V_{124} = 1$, we have
\[
\beta_{12} = - \langle v_3,v_4 \rangle |v_1 \times v_2| \cdot \frac{w_m(P)}{3 \alpha_4}.
\]
Multiplying both sides in (\ref{eq:isotropic_w}) with some other $v_i$s, we obtain that 
for any $i \in \{ 1,2\}$ and $j \in \{ 3,4 \}$, and with the notation $\{ i,j,s,t \} = \{ 1,2,3,4\}$, we have
\[
\beta_{ij} = - \langle v_s,v_t \rangle |v_i \times v_j| \cdot \frac{w_m(P)}{3 \alpha_6},
\]
and obtain similarly that $\langle v_1, v_2 \rangle = 0$.
 
Now we set $\bar{\beta}_{ij} = - \langle v_s, v_t \rangle |v_i \times v_j|$ for all $\{ i,j,s,t \} = \{ 1,2,3,4\}$, and note that
$\beta_{12}= \frac{w_m(P) \bar{\beta}_{12}}{3\alpha_4}$, $\beta_{34}= \zeta_{34}=0$, and for any $ij \neq 12, 34$, $\beta_{ij}= \frac{w_m(P) \bar{\beta}_{ij}}{3\alpha_6}$.
Substituting back these quantities into the formulas for $\vol_3(P)$ and $w_m(P)$ and simplifying, we may rewrite our optimization problem in the following form:
find the maximum of $\frac{27 \vol_3(P)}{\left(w_m(P) \right)^3} = f\left(\frac{\bar{\beta}_{12}}{\alpha_4}, \frac{\bar{\beta}_{13}}{\alpha_6}, \ldots, \frac{\bar{\beta}_{34}}{\alpha_6}, 0 \right)$, where $\bar{\beta}_{ij}$ is defined as above, in the family of all centered tetrahedra $T$ with $\vol_3(T)=\frac{2}{3}$, under the condition that $\sum_{1 \leq i < j \leq 4} \bar{\beta}_{ij} |v_i \times v_j| = 3$. Here, it is worth noting that the last condition is satisfied for any centered tetrahedron $T$ with volume $\frac{2}{3}$ by Lemma~\ref{lem:tetrahedron_identities}, and thus, it is redundant.
Note that since $f$ is $3$-homogeneous, we have
\[
f\left(\frac{\bar{\beta}_{12}}{\alpha_4}, \frac{\bar{\beta}_{13}}{\alpha_6} \ldots, \frac{\bar{\beta}_{24}}{\alpha_6}, 0 \right) =
\frac{1}{\alpha_6^3} f\left(\mu \bar{\beta}_{12}, \bar{\beta}_{13}, \ldots, \bar{\beta}_{24}, 0 \right),
\]
where $\mu = \frac{\alpha_6}{\alpha_4} \geq 1$.

In the second step we give an upper bound for the values of this function.
Set $\gamma_{ij} = - \langle v_s, v_t \rangle$ and $\zeta_{ij}= \gamma_{ij} |v_i \times v_j|^2$ for all $\{i,j,s,t\} = \{ 1,2,3,4\}$, and note that this yields $\gamma_{34}=\zeta_{34}=0$.
To find an upper bound, we apply the Cauchy-Schwartz Inequality, which states that for any real numbers $x_i,y_i$, $i=1,2,\ldots, k$, we have
$\sum_{i=1}^k x_i y_i \leq \sqrt{\sum_{i=1}^k x_i^2} \sqrt{\sum_{i=1}^k y_i^2}$, with equality if and only if the $x_i$s and the $y_i$s are proportional.
To do this, we write each member of $f$ not containing $\bar{\beta}_{12}$ as the product $\bar{\beta}_{ij} \bar{\beta}_{kl} \bar{\beta}_{mn} = \sqrt{\gamma_{ij} \gamma_{kl} \gamma_{mn}} \sqrt{\zeta_{ij} \zeta_{kl} \zeta_{mn}}$, and each member containing $\bar{\beta}_{12}$ as $ \mu \bar{\beta}_{12} \bar{\beta}_{kl} \bar{\beta}_{mn} = \sqrt{\gamma_{12} \gamma_{kl} \gamma_{mn}} \sqrt{\mu^2 \zeta_{12} \zeta_{kl} \zeta_{mn}}$. Thus, we obtain that
\[
f(\mu \bar{\beta}_{12},\bar{\beta}_{13},\ldots,\bar{\beta}_{24},0) \leq \sqrt{f(\mu^2 \zeta_{12},\zeta_{13},\zeta_{14},\zeta_{23},\zeta_{24},0)} \cdot
\sqrt{f(\gamma_{12},\gamma_{13},\gamma_{14},\gamma_{23},\gamma_{24},\gamma_{34})}.
\]
Here we use the fact that by Lemma~\ref{lem:tetrahedron_identities}, $f(\gamma_{12}, \ldots, \gamma_{34}) = 1$ for all centered tetrahedra with volume $\frac{2}{3}$.
Furthermore, observe that by Lemma~\ref{lem:tetrahedron_identities} we have $\sum_{1 \leq i \leq j \leq 4} \tau_{ij} = 3$ for all such tetrahedra. Hence, 
by Lemma~\ref{lem:computations2}, we have
\[
f(\mu^2 \zeta_{12}, \zeta_{13}, \ldots, \zeta_{24},0) \leq \frac{16 \mu^6}{(4\mu^2 - 1)^2},
\]
implying that if $\vol_3(P) = 1$, then
\[
w_m(P) \geq \frac{3 \alpha_4^{1/3} (4 \alpha_6^2 - \alpha_4^2)^{1/3}}{2^{2/3}}.
\]
Consequently, we have
\begin{equation}\label{eq:type4result}
w_m^4 \geq \left\{
\begin{array}{l}
w_m^5, \hbox{ if } \alpha_4 > \alpha_6;\\
\frac{3 \alpha_4^{1/3} (4 \alpha_6^2 - \alpha_4^2)^{1/3}}{2^{2/3}}, \hbox{ if } \alpha_4 \leq \alpha_6,
\end{array}
\right.
\end{equation}
and equality in the first case does not occur.

\textbf{Case 5}: $P$ is a type (5) parallelohedron.
In this case every generating segment generates a $6$-belt, and we can apply \cite[Theorem 1]{Langi} directly, and obtain that
as $\vol_3(P)=1$,
\[
w_m(P) \geq \frac{3}{2^{1/6}} \alpha_6,
\]
with equality if and only if $P$ is a regular truncated octahedron.

Table~\ref{tab:wmP} collects our lower bounds for $w_m(P)$ for the different types of parallelohedra. Note that if $i \neq 4$, $w_m^i$ coincides with the expression in the $i$th row and second column of Table~\ref{tab:wmP}, whereas $w_m^4$ is not less than the quantity in the $4$th row and the second column.

\begin{table}[h]
\begin{tabular}{>{\centering}m{0.1\textwidth}||>{\centering}m{0.25\textwidth}|>{\centering}m{0.5\textwidth}}
Type & $w_m^i$ & Optimal parallelohedra\tabularnewline
\hline \hline
(1) & $3\alpha_4$ & unit cube\tabularnewline
\hline
(2) & $\frac{3^{7/6}}{2^{1/3}} \alpha_4^{2/3} \alpha_6^{1/3} $ & regular hexagon based right prism, with base and lateral edges of lengths $\frac{2^{2/3} \alpha_6^{1/3}}{3^{5/6} \alpha_4^{1/3}}$ and $\frac{3^{1/6} \alpha_4^{2/3}}{2^{1/3} \alpha_6^{2/3}}$, respectively \tabularnewline
\hline
(3) & $2^{2/3} \cdot 3^{1/2} \alpha_6 $ & regular rhombic dodecahedron with edge length $\frac{\sqrt{3}}{2^{4/3}}$\tabularnewline
\hline
(4) & $\geq \frac{3 \alpha_4^{1/3} (4 \alpha_6^2 - \alpha_4^2)^{1/3}}{2^{2/3}} $, or $> w_m^5$ & not known \tabularnewline
\hline
(5) & $\frac{3}{2^{1/6}} \alpha_6$ & regular truncated octahedron of edge length $\frac{1}{2^{7/6}}$
\end{tabular}
\caption{Minima of $w_m(P)$ for different types of parallelohedra $P$}
\label{tab:wmP}
\end{table}

In the final step we determine the minimum of $\{ w_m^i : i=1,2,3,4,5\}$ as a function of $\alpha_4$ and $\alpha_6$.

\begin{figure}[ht]
\begin{center}
\includegraphics[width=0.5\textwidth]{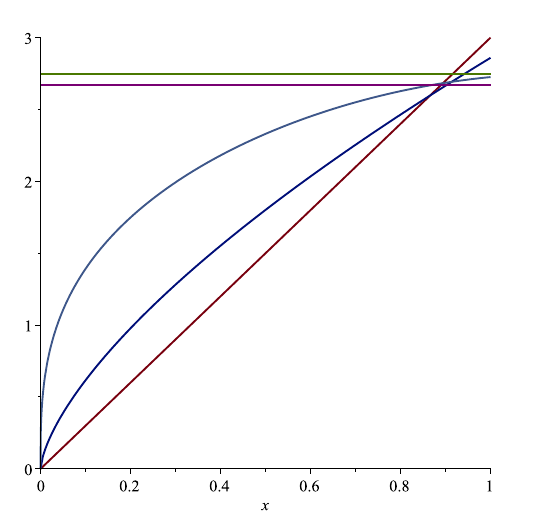}
\caption{The lower bounds for $w_m^i$ from the second column of Table~\ref{tab:wmP} as functions of $\alpha_4$, with $\alpha_6 = 1$, and $0 < \alpha_4 \leq 1$.}
\label{fig:graphs}
\end{center}
\end{figure}

An elementary computation shows that $w_m^3 > w_m^5$ for all $\alpha_4,\alpha_6 > 0$, and that
$w_m^4 \geq \frac{3 \alpha_4^{1/3} (4 \alpha_6^2 - \alpha_4^2)^{1/3}}{2^{2/3}} > \min \{ w_m^1, w_m^5 \}$ if
$0 < \alpha_4 \leq \alpha_6$. Furthermore, if $\alpha_4 > \alpha_6$, then $w_m^4 > w_m^5$.
Thus, we need to compare only the values of $w_m^1$, $w_m^2$ and $w_m^5$, and from this another elementary computation proves the statement.

\end{document}